\documentclass[12pt,reqno]{amsart}
\usepackage[utf8]{inputenc}
\usepackage[top=1in, bottom=1in, left=1in, right=1in]{geometry}
\usepackage{times, amsthm, amssymb, amsmath, amsfonts, graphicx, mathrsfs}
\usepackage{mathtools}
\usepackage{array}
\usepackage{tikz} 
\usetikzlibrary{arrows, cd, matrix,positioning}
\usepackage{xcolor}
\usepackage{bbm}
\usepackage{xypic}
\usepackage{mathscinet}
\usepackage{algorithm}
\usepackage{algpseudocode}
\usepackage{hyperref}

\newtheorem{theorem}{Theorem}[section]
\newtheorem{lemma}[theorem]{Lemma}
\newtheorem{proposition}[theorem]{Proposition}
\newtheorem{corollary}[theorem]{Corollary}
\newtheorem{question}[theorem]{Question}
\newtheorem{conjecture}[theorem]{Conjecture}

\theoremstyle{definition}
\newtheorem{definition}[theorem]{Definition}
\newtheorem{example}[theorem]{Example}
 
\theoremstyle{remark}
\newtheorem{remark}[theorem]{Remark}

\newtheorem*{notation}{Notation}

\DeclareMathOperator{\Ima}{Im}

\newcommand{\height}{\operatorname{height}}

\newcommand{\Ann}{\operatorname{Ann}}
\newcommand{\Tor}{\operatorname{Tor}}

\newcommand{\DM}{\operatorname{DM}}
\newcommand{\DDM}{\operatorname{D_{DM}}}
\newcommand{\bDDM}{\operatorname{D^b_{DM}}}
\newcommand{\ch}{\operatorname{Ch}}
\newcommand{\fold}{\operatorname{Fold}}
\newcommand{\Db}{\operatorname{D^b}}

\newcommand{\fclass}{\operatorname{free class}}

\newcommand{\bbA}{\mathbbm{A}}
\newcommand{\bbB}{\mathbbm{B}}

\newcommand{\bbN}{\mathbbm{N}}

\newcommand{\bbP}{\mathbbm{P}}
\newcommand{\bbQ}{\mathbbm{Q}}

\newcommand{\bbV}{\mathbbm{V}}

\newcommand{\bbZ}{\mathbbm{Z}}

\newcommand{\bbk}{\mathbbm{k}}

\newcommand{\calC}{{\mathcal{C}}}
\newcommand{\calD}{{\mathcal{D}}}
\newcommand{\calE}{{\mathcal{E}}}
\newcommand{\calF}{{\mathcal{F}}}

\newcommand{\calO}{{\mathcal{O}}}

\newcommand{\frakm}{{\mathfrak{m}}}

\newcommand{\bfF}{\mathbf{F}}

\newcommand{\bfK}{\mathbf{K}}

\newcommand{\bfa}{\mathbf{a}}
\newcommand{\bfb}{\mathbf{b}}

\newcommand{\bfe}{\mathbf{e}}

\newcommand{\del}{{\partial}}

\newcommand{\betadm}{\beta^{\text{DM}}}
\newcommand{\bsmod}[1]{BS_{\text{mod}}(#1)}
\newcommand{\bsdm}[2]{BS_{\text{DM}}(#1, #2)}
\newcommand{\cvb}[1]{C_{\text{vb}}(#1)}

\renewcommand{\th}{{^\text{th}}}

\newcommand{\bs}{\text{Boij-S\"oderberg }}

\title{Boij-S\"oderberg Conjectures for Differential Modules}
\author{Maya Banks}
\date{\today}

\begin{document}

\maketitle
\begin{abstract}
Boij-S\"oderberg theory gives a combinatorial description of the set of Betti tables belonging to finite length modules over the polynomial ring $S = \bbk[x_1, \ldots, x_n]$. We posit that a similar combinatorial description can be given for analogous numerical invariants of \emph{graded differential $S$-modules}, which are natural generalizations of chain complexes. We prove several results that lend evidence in support of this conjecture, including a categorical pairing between the derived categories of graded differential $S$-modules and coherent sheaves on $\bbP^{n-1}$ and a proof of the conjecture in the case where $S = \bbk[t]$.
\end{abstract}

\section{Introduction}
\emph{Differential modules}, modules equipped with an endomorphism squaring to zero, are natural generalizations of chain complexes. One enlightening way to picture a differential module is as a complex absent the structure of a homological grading. Initially appearing in more of a book-keeping capacity in \cite{cartan2016homological}, differential modules have increasingly become objects of interest in commutative algebra, algebraic geometry, and algebraic topology in particular as useful tools for generalizing classical results and conjectures \cite{csenturk2019carlsson, iyengar2018examples, ToricTate, avramov2007class, brown2021minimal, boocher2010rank}. An overarching question in recent years can be distilled down to the following.
\begin{question}\label{q:mainq}
When can results about familiar objects be generalized to results about differential modules,  and what insight does the ability or failure to generalize yield about the original results?
\end{question}
Avramov-Buchweitz-Iyengar \cite{avramov2007class} investigate this question for a special class of differential modules---\emph{flag differential modules}---and find generalizations of several results in commutative algebra, including a generalization of the New Intersection Theorem of Hochster, Peskine, Roberts, and Szpiro \cite{Hoc74,PS73,Rob89}. Brown and Erman give a generalization of the notion of a minimal free resolution \cite{brown2021minimal}, and Banks and VandeBogert construct a differential module generalization of the Koszul complex \cite{BV22}. On the other hand, \cite{BV22} also provide examples of the failure of certain notions to generalize nicely to the realm of differential modules, while Iyengar and Walker similarly show that bounds on the Betti numbers of a free resolution do not generalize to differential modules even in the case of free flags \cite{iyengar2018examples}.

Our goal in this paper is to explore Question \ref{q:mainq} for Boij-S\"oderberg theory, which concerns the combinatorial structure of the set of all possible Betti tables of modules over a given ring. Initially conjectured by Boij and S\"oderberg \cite{2008} for modules over the polynomial ring, this theory brought a new perspective to the study of Betti tables.
 The \bs conjectures were proven in the finite length case in 2008 by Eisenbud and Schreyer via a duality between free resolutions of graded modules over the polynomial ring and vector bundles on projective space \cite{eisenbud2008Betti}. Their results describe the positive rational cone spanned by Betti tables belonging to modules over the polynomial ring and give a classification, up to rational multiple, of all such tables. 

 \begin{remark}
     The original theory was developed in the finite length case and subsequently extended to other codimensions. For simplicity, we restrict ourselves here to the finite length case as well.
 \end{remark}

The results of \cite{eisenbud2008Betti, boij2012betti} were soon generalized and extended in various directions, such as \cite{eisenbud2009cohomology, 2009, Floystad_2010, 2011, Berkesch_2011, Berkesch_2012, Boij_2015, Erman_2016, Ford_2018a, Ford_2018b, DeStefani2021} as well as most recently in \cite{IMW22}. In a particularly interesting extension of the original theory, Eisenbud and Erman \cite{EE2017} showed that the duality in \cite{eisenbud2008Betti} does not actually depend on the structure of a free resolution---the duality holds for minimal free complexes with finite length homology more generally (a theme further explored by \cite{IMW22}). This allowed the original Boij-S\"oderberg theory to be extended from free resolutions to the more general setting of free complexes, but it also showed that the foundations of the theory rested on a different level of structure than what was originally assumed and conjectured. In particular, one interpretation of this generalization is that we should think of Boij-S\"oderberg theory as \emph{actually} being a theory about the numerics of minimal finite free complexes, rather than a theory about invariants of modules. This leads us to question what the ``correct" level of generality for this theory really is. That is, how far can the original Boij-S\"oderberg theory extend, and what is the key structure on which the theory truly relies? Can Boij-S\"oderberg theory be further generalized beyond free complexes, and what new insight might such a generalization give us?

We posit a generalization of the Boij-S\"oderberg conjectures for differential modules and build evidence in support of this conjecture, including some partial results that shed new light on some of the key structures underlying the original theory, while simultaneously giving insight on what makes differential modules similar to and different from complexes and free resolutions.

Let $S = \bbk[x_1, \ldots, x_n]$ for any field $\bbk$ and denote by $\DM(S,a)$ the category of differential modules of degree $a$ over $S$, that is the graded modules $D$ over $S$ with a square 0 endomorphism $D\to D(a)$. 
For a differential module $D\in \DM(S,a)$, the \emph{Betti vector} of $D$, denoted $\betadm(D)$ is the vector whose $i\th$ entry counts the number of degree $i$ generators in a minimal free resolution of $D$ (see Definition \ref{def:min}). 

Let $\bbV = \bigoplus_{i,j\in\bbZ}\bbQ$ and $\bbB = \bigoplus_{j\in\bbZ}\bbQ$, so that Betti tables and Betti vectors are naturally elements of $\bbV$ and $\bbB$, respectively. We denote by $\bsmod{S}\subset\bbV$ and $\bsdm{S}{a}\subset\bbB$ the rational cones of Betti tables of graded, finite length $S$-modules and Betti vectors of degree $a$ differential $S$-modules with finite length homology.

The Betti vector of a differential module is an analogous numerical object to the Betti table of a minimal free resolution, but it represents a certain coarsening of the data of a Betti table. More precisely, for any $a\in\bbZ$ we may regard any $S$-complex as a differential module in $\DM(S,a)$ via what \cite{brown2021minimal} call the \emph{folding functor}, which we define in more detail in Section 2. As an operation on Betti tables, the folding functor ``flattens" the Betti table of a minimal free resolution by forgetting all homological data and retaining only the internal grading. In other words, this flattening is a map $\bbV\to\bbB$ that maps Betti tables of minimal free resolutions of $S$-modules into $\bsdm{S}{a}$ for any $a$. We propose the following:

\begin{conjecture}\label{conj:intro}
This ``flattening" gives a surjection of cones $\bsmod{S}\to\bsdm{S}{a}$ for every $a$. More precisely Any differential module with finite length homology has a Betti vector that can be written as a positive rational combination of Betti vectors of differential modules whose homology have Betti tables that are extremal in $\bsmod{S}$.
\end{conjecture}

Before moving on to our results in support of this conjecture, we first consider a small example

\begin{example}\label{ex:intro}
    Consider the two variable case when $S = \bbk[x,y]$ and $(D, \del)$ is the degree 0 differential module where

\begin{align*}
    D = S\oplus S(-1)^2 \oplus S(-2) \text{ and } \del = \begin{bmatrix}
           0 & x & y & 0\\
           0 & 0 & 0 & -y\\
           0 & 0 & 0 & x\\
           0 & 0 & 0 & 0
    \end{bmatrix}
\end{align*}
This differential module has homology $\bbk$ and Betti vector $\betadm(D) = (1^\circ, 2, 1)$ where $1^\circ$ indicates that the 1 is in the degree 0 position in the vector. In fact, by \cite{BV22}, every degree 0 differential module over $S = \bbk[x,y]$ with homology $\bbk$ has this Betti vector.

On the other hand, suppose we instead let $(D,\del)$ be a degree 2 differential module over $S$, with

\begin{align*}
    D = S\oplus S(1)^2 \oplus S(2) \text{ and } \del = \begin{bmatrix}
           0 & x & y & 0\\
           0 & 0 & 0 & -y\\
           0 & 0 & 0 & x\\
           0 & 0 & 0 & 0
    \end{bmatrix}
\end{align*}
Again, $D$ has homology $\bbk$ and now its Betti vector is $\betadm(D) = (1,2,1^\circ)$.
On the other hand, let $(D', \del')$ be the degree 2 differential module where
\begin{align*}
    D' = S\oplus S(1)^2 \oplus S(2) \text{ and } \del = \begin{bmatrix}
           0 & x & y & 1\\
           0 & 0 & 0 & -y\\
           0 & 0 & 0 & x\\
           0 & 0 & 0 & 0
    \end{bmatrix}
\end{align*}

This differential module also has homology $\bbk$ but, unlike $(D,\del)$, $(D', \del')$ is not minimal. It has minimal free resolution 
\[
\begin{bmatrix}
       xy & x^2\\
       -y^2 & -xy
\end{bmatrix}
\]
and Betti vector $\betadm(D') = (0, 2, 0^\circ)$.
\end{example}

We draw attention to three particular features of this example and the implications that they highlight. First, observe that $\betadm(D')$ cannot be the Betti vector of a fold of minimal free resolution of a finite length, graded $S$-module since any such module must have total Betti numbers at least 4. In other words, while we conjecture that flattening gives a surjection on \emph{cones}, this example shows that it certainly cannot be a surjection onto the actual semigroup of Betti vectors. The second observation is that we have two differential modules whose homologies are isomorphic, but whose minimal free resolutions--and thus their Betti vectors--are genuinely different. This behavior represents a stark departure from what we are used to in the realm of finite length, graded $S$-modules. Finally, we point out the differing behavior exhibited by the differential modules in $\DM(S,a)$ as $a$ varies from 0 to 2. 

This third point is actually a more general phenomenon that will continue to appear. When dealing with graded complexes, we are always free to twist the individual modules in the complex so that the differential is of degree 0. For differential modules however, the source and target of our differential are the same, so we can not twist one without twisting the other. This means that we are unable to control the degree of the differential in the way that we are used to, and must contend with differentials of nonzero degree. While a differential of nonzero degree may seem innocuous at first glance, we will see that this variance can actually have drastic impacts on the numerical invariants at hand. This phenomenon was observed as well in \cite[Example 5.4]{brown2021minimal} and \cite[Corollary 3.7, Proposition 3.8]{BV22}, and turns out to be a theme underpinning our results here as well. In general, it seems that while for degree 0 we may expect differential modules to behave more or less similarly to complexes in many regards, the same is not the case in more general degree. For Boij-S\"oderberg theory specifically, this behavior seems to support the idea that the ability to force the differential to be degree 0 is actually an indispensable piece of the structure on which the theory rests. More broadly, it seems that one answer to the question of what makes differential modules similar to or different from complexes lies in their degree.

\subsection{Results} 

We split our results into three cases depending on whether the degree of the differential module is positive, negative, or zero. The positive degree case turns out to be much more straightforward---as well as much less reminiscent of the standard Boij-S\"oderberg theory---and we are able prove the following in full generality via quite elementary methods.

\begin{theorem}\label{thm:intro-a>0}
    Conjecture \ref{conj:intro} is true for differential modules of degree $a>0$.
\end{theorem}

We are able to prove this result by explicitly constructing all of the extremal rays of the Betti cone. The proof relies on a certain level of ``cancellation" that can appear when passing from a differential module to its minimal free resolution--as demonstrated in Example \ref{ex:intro}. Our next result says essentially that when the degree of the differential module is not positive, there is a limit to this cancellation.

We prove the following which we show arises as a consequence of Avramov, Buchweitz, and Iyengar's Class Inequality \cite[Theorem 4.1]{avramov2007class}.

\begin{proposition}\label{prop:intro-n+1}
    Let $F\in \DM(S,a)$ be a free differential module with finite length homology, and assume that $a\leq 0$. Then the Betti vector $\betadm(F)$ has at least $n+1$ nonzero entries.
\end{proposition}

Recall that in standard Boij-S\"oderberg theory the extremal rays of the Betti cone come from the Betti tables of modules with \emph{pure resolutions}, those with the fewest possible nonzero entries. Over $S = \bbk[x_1, \ldots, x_n]$, Betti tables of pure resolutions have exactly $n+1$ nonzero entries. 
Since folds of pure resolutions in degree $a\leq 0$ have Betti vectors with exactly $n+1$ nonzero entries, Proposition \ref{prop:intro-n+1} gives a hint that the extremal rays of $\bsmod{S}$ might remain extremal after flattening. This would mean that the extremal rays in $\bsdm{S}{a}$ come from analogous objects as in the standard theory. This provides some motivation for Conjecture \ref{conj:intro} in the $a\leq 0$ case as well as some hope that familiar methods may be applicable.

Indeed, when we restrict to the case $a = 0$, we are able to make use of familiar methods by generalizing the categorical pairing of Eisenbud and Erman \cite{EE2017}. The categorified duality in \cite{EE2017} clarified the pairing between Betti tables of finite length $S$-modules and cohomology tables of vector bundles on $\bbP^{n-1}$ which played a crucial role in the proof of the original Boij-S\"oderberg conjectures \cite{eisenbud2008Betti}. In the original proof, Eisenbud and Schreyer show a sort of duality between $\bsmod{S}$ and the cone $C_{\text{vb}}(\bbP^{n-1})$ of cohomology tables of vector bundles on $\bbP^{n-1}$, which they then use to prove that the facet equations of the cone spanned by the pure Betti tables are nonnegative on the Betti table of any finite length $S$-module. Eisenbud and Erman then showed that there was in fact a pairing

\[
\bsmod{S}\times \cvb{\bbP^{n-1}}\to \bsmod{\bbk[t]}
\]

which captured this duality.

We generalize this pairing to the setting of degree 0 differential modules and prove analogs of two key results from \cite{EE2017}.
In particular, we construct a functor 
\[
\Phi \colon \bDDM(S,0)\times \Db(\bbP^{n-1})\to\bDDM(A,0)
\]
where $A = \bbk[t]$, $\bDDM(S,0)$ and $\bDDM(A,0)$ denote the bounded derived categories of degree 0 differential modules over $S$ and $A$, and $\Db(\bbP^{n-1})$ denotes the bounded derived category of coherent sheaves on $\bbP^{n-1}$. We show that $\Phi$ satisfies two crucial properties.

\begin{theorem}\label{thm:intro2}
Let $F\in \bDDM(S,0)$ be free and $\calE\in \Db(\bbP^{n-1})$. Then \begin{enumerate}
    \item The Betti vector of $\Phi(F,\calE)$ depends only on the Betti vector of $F$ and the absolute Hilbert function (see Definition \ref{def:absHilb}) of $\calE$.
    \item If $\tilde{F}\otimes\calE$ is exact, then $\Phi(F,\calE)$ has finite length homology.
\end{enumerate}
\end{theorem}

Thus our $\Phi$ also induces a pairing 
\[
\bsdm{S}{0}\times\cvb{\bbP^{n-1}}\to\bsdm{A}{0}.
\]


The existence of this pairing leads us to conjecture the following analog of the duality displayed in \cite[Theorem 0.2]{EE2017}.

\begin{conjecture}
    Let $\bfb\in \bbB$. Then the following are equivalent:
    \begin{enumerate}
        \item $\bfb = \betadm(F)$ for a degree 0 differential module $F$ over $S$ with finite length homology.
        \item For every vector bundle $\calE$ on $\bbP^{n-1}$, $\bfb$ pairs with the absolute Hilbert function of $\calE$ to give the Betti vector of a differential module in $\DM(A,0)$ with finite length homology.
    \end{enumerate}
\end{conjecture}

Finally, we prove that Conjecture \ref{conj:intro} holds for degree 0 differential modules over $A$, explicitly:

\begin{theorem}\label{thm:intro3}
Every degree 0 differential modules over $A = \bbk[t]$ with finite length homology has a Betti vector that can be expressed as a positive rational combination of Betti vectors of folds of pure resolutions whose degree sequences form a chain.
\end{theorem}

This result is part of what gives the categorical pairing its power. With a complete description of the extremal rays and facets of the Betti cone over $A$, the pairing $\Phi$ gives a machine for proving that a certain class of linear functionals are nonnegative on every Betti vector in $\bsdm{S}{0}$ by reducing to $\bsdm{A}{0}$ as a sort of ``base case". Proving this nonnegativity is one of the main steps in a proof of Conjecture \ref{conj:dmbs} in the general case. 
What ultimately remains standing in the way of a full proof of Conjecture \ref{conj:intro} for degree 0 is a combinatorial description of the cone spanned by the pure Betti vectors. The authors of \cite{eisenbud2008Betti} describe a simplicial fan structure which gives a facet description of the cone for finite length $S$-modules. While we do this for our base case of $\bsdm{A}{0}$, we have not yet been able to replicate this for $\bsdm{S}{0}$ more generally, though computational data enables us to conjecture that the extremal facets arise similarly as in \cite{EE2017}. Once such a description is proven, we can use the categorical pairing to prove that the equations of these facets are actually nonnegative on all Betti vectors of degree 0 differential modules with finite length homology.

In the case where $a<0$, we are forced to abandon some of the methods that generalized nicely from complexes to degree 0 differential modules, such as the categorical pairing. However, the positive results we have obtained in the degree $\geq 0$ cases lead us to believe that it is natural to extend Conjecture \ref{conj:intro} to the negative degree case as well. That being said, in this setting, new methods appear to be needed to tackle the conjecture and computational data has thus far proved perplexing.

\subsection*{Acknowledgements} The author would like to thank Daniel Erman for invaluable guidance, suggestions, and feedback throughout this work, and Michael Brown for helpful conversations contributing to this project. Additionally, the computer algebra system Macaulay2 \cite{M2} provided crucial insight and examples.

\section{Background}

For the rest of the paper, our convention will be to assume that $R$ is graded-local with maximal ideal $\frakm$ and residue field $\bbk = R/\frakm$.

\begin{definition}
For a ring $R$, a \emph{differential $R$-module} is a pair $(D,\del)$ where $D$ is a module over $R$ and $\del\colon D\to D$ is an $R$-module endomorphism squaring to 0. In the case where $D$ is graded, we may define, for any $a\in \bbZ$, a \emph{degree $a$ differential module} over $R$ to be a differential module $(D,\del)$ where $\del\colon D\to D(a)$ is homogeneous. We denote the category of degree $a$ differential modules over $R$ by $\DM(R,a)$. A \emph{morphism} of differential modules is a module map that respects the differential, that is a map $f\colon (D,\del)\to (D',\del')$ satisfying $f\del = \del'f$.
\end{definition}

Note that the data of a differential $R$-module $(D,\del)$ is equivalent to the data of a graded module over the ring $R[\varepsilon]/(\varepsilon^2)$ with $\deg(\varepsilon) = a$, where the action of $\varepsilon$ corresponds to the action of the differential $\del$. This implies that $\DM(R,a)$ is an abelian category.

\begin{definition}
For a differential module $(D,\del)\in \DM(R,a)$, the \emph{homology} of $D$ is \[H(D) = \ker(\del)/\Ima(\del)(-a)\]. A morphism of differential modules is a \emph{quasi-isomorphism} if it induces an isomorphism on homology.
\end{definition}

Differential modules are in a sense a generalization of chain complexes. In particular, any chain complex may be considered as a differential module where the underlying module is the direct sum of the modules in the complex and the differential is the differential from the complex. More precisely, we have a functor 

\begin{align*}
\fold_a\colon &\ch(R)\to \DM(R,a)\\
& (C_\bullet, \del)\mapsto \left(\bigoplus C_i(ia), \del \right).
\end{align*}

In general, there is no way to consistently define a tensor product of two differential modules. However, we can define the tensor product of a differential module with a \emph{complex} as follows. Given a differential module $D\in \DM(R,a)$ and a graded degree $a$ complex $C_\bullet$ of $R,R'$-bimodules, the tensor product $D\otimes^{DM}_R C_\bullet\in \DM(R',a)$ is the differential module whose underlying module is $\fold_a(D\otimes_R C_\bullet)$ and whose differential is given by 
\[
d\otimes c\mapsto d\otimes \del^C(c) + (-1)^{\deg(c)}c\otimes \del^D(d).
\]

Since we may regard any $R$-module as a complex concentrated in degree 0, this defines a tensor product $D\otimes^{DM}_R N$ for any $R$-module $N$.

While a differential module is a generalization of a chain complex, passing from complexes to more general differential modules destroys the homological grading which is key to proving many results for complexes. However, we are able to regain some of the benefits of a homological grading for certain differential modules defined below.

\begin{definition}
A \emph{flag} is a differential module $(F, \del)$ of the form $F = \bigoplus F_i$ where \[\del(F_i)\subset \bigoplus_{j<i}F_j\]. A flag $F$ is \emph{free/projective} if each of the $F_i$ is free/projective. For a differential module $F$, the \emph{free/projective class} of $F$ is 
\[
\min \left\{n\in\bbN \text{ such that } F \text{ admits a free/projective flag structure } F = \bigoplus_{i=0}^n F_i\right\}.
\]
The free class of $F$ over a ring $R$ is denoted $\fclass_R{F}$. If no flag structure on $F$ exists, then we say $\fclass_R{F} = \infty$
\end{definition}

\begin{theorem}[\cite{avramov2007class} Class Inequality]\label{thm:classineq}
Let $D$ be a finitely generated differential module over the polynomial ring $S = \bbk[x_1, \ldots, x_n]$. Then
\[
\fclass_S{D}\geq\height(\Ann H(D)).
\]
\end{theorem}


Given a (free) flag $F$, we can always extract certain complexes from $F$. First observe that a free flag $F$ over $R$ consists of the data of a family of free $R$-modules $\{F_i\}_{i\in\bbN}$ and a family of morphisms $\{A_{i,j}\colon F_i\to F_j\}_{i>j}$ such that 
\[
\sum_{i>j>k} A_{j,k}A_{i,j} = 0.
\]
In particular, this requires that $A_{i,i-1}A_{i+1,i} = 0$ for all $i\geq 1$, so we there is a free complex
\[
\bfF = \quad F_0\xleftarrow{A_{1,0}} F_1\xleftarrow{A_{2,1}} F_2\leftarrow \cdots
\]
\begin{definition} In the setting detailed above, we say that the flag $\left(\bigoplus_{i\geq 0} F_i, \{A_{i,j}\}\right)$ is \emph{anchored by} the complex $\bfF$. This definition was given in \cite{BV22}.
\end{definition}

One main goal in the study of differential modules has been to generalize constructions and results for complexes to the differential module setting. Key progress in this endeavor was made by Brown and Erman \cite{brown2021minimal} in their definition of free flag and minimal free resolutions of differential modules, which we reiterate here.

\begin{definition}
Let $D$ be a differential module over a ring $R$. A \emph{free flag resolution} of $D$ is a quasi-isomorphism $F\to D$ where $F$ is a free flag.
\end{definition}

A result of \cite{brown2021minimal} highlights an important link between a differential module $D$ and the minimal free resolution of its homology. Their theorem says that \emph{every} differential module has a free flag resolution anchored on the minimal free resolution of its homology. Following the terminology of \cite{BV22}, we call these free flag resolutions \emph{anchored free flags}.



\begin{definition}\label{def:min}
We say a differential $R$-module $(M,\del)$ is \emph{minimal} if $\del(M)\subseteq \frakm M$. For a differential $R$-module $D$,   \emph{minimal free resolution} of $D$ is a quasi-isomorphism $M\to D$ that factors through a splitting map $M\to F$ where $F$ is a free flag resolution of $D$ and $M$ is minimal.
\end{definition}

It is a further result of \cite{brown2021minimal} that for differential modules admitting a finitely generated free flag resolution, minimal free resolutions exist, are finitely generated, and are unique up to differential module isomorphism.

We can define certain numerical invariants associated to a differential module. In particular, we have a parallel notion of the Betti numbers of graded differential module:

\begin{definition}
Let $D$ be a differential $R$-module with minimal free resolution $F$ and let $N$ be an $R$-module. Then we define 
\[
\Tor_{DM}^R(D,N) := H(F\otimes_R^{DM} N).
\]
The \emph{Betti numbers} of $D$ are then defined as 
\[
\betadm_j(D) = \dim_{\bbk}\Tor_{DM}^R(D, \bbk)_j
\]
\end{definition}

\begin{remark}\label{rmk:degjgens}
One important observation is that, if $M$ is a minimal free resolution of $D$, then $\betadm_j(D)$ is equal to the number of generators of the degree $j$ part of $M$ as a $\bbk$-vector space (see \cite[Remark 1.3(4)]{brown2021minimal}).
\end{remark}

Recall that a Betti diagram is called \emph{pure} if it has at most one nonzero entry in each column, and that a resolution with a pure Betti diagram is called a \emph{pure resolution}. We say that a Betti vector is pure if it is the Betti vector of the fold of a pure resolution. 

Each pure Betti vector has an associated \emph{degree sequence}, an increasing sequence of integers corresponding to the degrees of the nonzero entries in the Betti vector. We put a partial order on the set of degree sequences of length $\ell$ by saying that $\bfa = (a_0<a_1<\cdots < a_\ell)\leq \bfb = (b_0<b_1<\cdots < b_\ell)$ if $a_i\leq b_i$ for every $i$.

\section{From Free Complexes to Differential Modules}
In this section, we compare the theory of Betti tables of free resolutions and free complexes to that of Betti vectors of differential modules and propose a generalization of the Boij-S\"oderberg conjectures for differential modules.\

Let $S = \bbk[x_1, \ldots, x_n]$ for any fixed field $\bbk$. We denote the space of \emph{rational Betti vectors} by $\bbB = \bigoplus_{j\in\bbZ}\bbQ$. As in the case of the original Boij-S\"oderberg theory, our goal is to describe the rational cone spanned by the vectors arising as the Betti vectors of differential modules with finite length homology. 

We begin by exploring how Betti vectors of differential modules over $S$ are related to Betti tables of free $S$-complexes. First note that by Remark \ref{rmk:degjgens} it is enough to consider minimal free resolutions of differential modules, and by \cite[Theorem 3.2, Theorem 4.2]{brown2021minimal} we may assume that up to differential module isomorphism any minimal free resolution of a differential module $D$ is a summand of an anchored free flag resolution $F\to D$ where $F$ is the fold of the minimal free resolution of the homology $H(D)$. We thus have that $\betadm_j(D)$ is at most the number of degree $j$ generators of $F$, which is in turn the sum $\sum_{i}\beta_{i,j}(H(D))$. 

\begin{definition}
Let $(\beta_{i,j})$ be the Betti table of a module $M$ over $S$. The \emph{degree $a$ flattening} of $(\beta_{i,j})$ is the vector $\bfb\in \bigoplus_{j\in\bbZ}\bbQ$ defined by 
\[
\bfb_j = \sum_{i} \beta_{i,ai+j}
\]
\end{definition}

\begin{remark}\label{rmk:foldflat}
Combining the definitions of flattening and folding, we get that for $\bfF$ a minimal free resolution of $M$, the Betti vector of  $~\fold_a(\bfF)$ is the degree $a$ flattening of the Betti table $(\beta_{i,j})$ of $M$.
\end{remark}

Remark \ref{rmk:foldflat} says that for certain differential modules $D$, the Betti vector is equal to the flattening of the Betti table of the $H(D)$. While this is always the case when $D$ is quasi-isomorphic to its homology, it is not true for all differential modules. We saw this in Example \ref{ex:intro} with a degree 2 differential module, but such behavior also occurs in other degrees, as we see with the following.

\begin{example}\label{ex:cancel} 
Let $R = \bbk[x,y]$ and let $D\in\DM(R,0)$ have underlying module 
\[
R\oplus R(-2) \oplus R(-1)^2\oplus R(-3)^2\oplus R(-2)\oplus R(-4)
\]
and differential given by the matrix
\[
\begin{bmatrix}
       0 & 0 & x & y & 0 & 0 & 0 & 0\\
       0 & 0 & 0 & 0 & x & y & 1 & 0\\
       0 & 0 & 0 & 0 & 0 & 0 & -y & 0\\
       0 & 0 & 0 & 0 & 0 & 0 & x & 0\\
       0 & 0 & 0 & 0 & 0 & 0 & 0 & -y\\
       0 & 0 & 0 & 0 & 0 & 0 & 0 & x\\
       0 & 0 & 0 & 0 & 0 & 0 & 0 & 0\\
       0 & 0 & 0 & 0 & 0 & 0 & 0 & 0
    \end{bmatrix}.
\]
The homology $H(D)$ is isomorphic to $\bbk\oplus\bbk(-2)$ which has Betti table 
\[
\begin{matrix}
       &0&1&2\\\text{0:}&1&2&1\\\text{1:}&\text{-}&\text
       {-}&\text{-}\\\text{2:}&1&2&1\\\end{matrix}
\]
This Betti table has degree 0 flattening $(1,2,2,2,1,0,\ldots)$, but the differential module $D$ is not minimal, so this is not the Betti vector of $D$. Rather, $D$ has Betti vector $(1,2,0,2,1)$. This reflects some ``cancellation" of the $R(-2)$ summands that occurs when passing from $D$ to its minimal free resolution.
\end{example}

\begin{remark}
In fact, a careful analysis (which we omit here) shows that there is no finite length graded module over $\bbk[x,y]$ whose Betti table flattens to $(1,2,0,2,1)$. However, there is a free \emph{complex} with finite length homology that gives this Betti vector. The complex in question has homology $\bbk$ in degree 0 and $\bbk(-2)$ in degree 1. Furthermore, the fold of this complex is isomorphic to the minimal free resolution of the differential module in the example. We note also that there \emph{is} a module whose Betti table flattens to a rational multiple of the vector $(1,2,0,2,1)$, so this vector lives on a ray contained in the Betti cone for finite length $S$-modules.
\end{remark}

Example \ref{ex:cancel} shows that there can be some difference between the Betti vector of a differential module and the Betti vector of the fold of the minimal free resolution of its homology. To understand the Betti cone of differential modules, we want to know how big this difference can be, that is, how much ``cancellation" can occur.
It turns out that the answer to this question depends substantially on the degree of the differential module, with the cases of positive degree and nonpositive degree exhibiting quite different properties. We thus divide the remainder of our discussion into those two cases.

\subsection{Differential modules of degree $a\leq 0$}
We begin with the nonpositive case, which we will see is the case that more closely resembles the behavior of complexes and resolutions. The following result says, in essence, that there is a limit to how much cancellation can occur in the Betti vector so long as the degree of the differential module is not positive.

\begin{proposition}\label{prop:n+1degs}
Let $F\in\DM(S,a)$ be a free differential module with finite length homology for some $a\leq 0$. Then $F$ is generated in at least $n+1$ distinct degrees.
\end{proposition}

\begin{proof}Let $(F,\del)$ be a minimal, free, graded differential module generated in degrees $d_0, \ldots, d_\ell$ and let $H = H(F)$, the homology of $F$. Then the grading gives a natural flag structure $F = \bigoplus_i F_i$ where $F_i$ is the piece of $F$ generated in degree at least $d_i$. This defines a flag structure on $F$, since minimality and nonpositivity of $a$ ensures that $\del(F_i)\subset \bigoplus_{i<j}F_j$. Applying the Class Inequality (Theorem \ref{thm:classineq}), $\ell\geq \height \Ann(H)$ which is equal to $n$ by the finite length assumption. So $F$ must be generated in at least $n+1$ distinct degrees.
\end{proof}

Note that in the case of finite length modules over $S$, any Betti table with exactly $n+1$ nonzero entries is the Betti table of a module with a pure resolution and, furthermore, Betti table is determined up to scalar multiple by the choice of nonzero entries. What we have shown for differential modules is weaker than this, but we conjecture that a similar phenomenon holds. That is, based on experimentation, we guess that the choice of nonzero entries determines the Betti vector up to scalar multiple, and that any Betti vector with exactly $n+1$ nonzero entries belongs to a differential module with pure homology. The reason this is nontrivial is that, due to the type of cancellation we have observed in Betti vectors, one could imagine a differential module $D$ with non-pure homology $H$ where some of the nonzero entries in the Betti table for $H$ cancel out, causing the number of nonzero entries in $\betadm(D)$ to be exactly $n+1$.

\begin{remark}\label{rmk:apos}
Proposition \ref{prop:n+1degs} relies very heavily on the degree $\leq 0$ condition. We have already seen that for differential modules in degree $a> 0$ it may fail spectacularly. We will return to this case at the end of this section.
\end{remark}

In degree $\leq 0$, Proposition \ref{prop:n+1degs} tells us that differential modules with finite length homology cannot be generated in fewer degrees than minimal free complexes with finite length homology. In particular, their Betti vectors have at least as many nonzero entries as the degree 0 flattenings of Betti tables of pure resolutions. On the other hand, we know that any flattening of a Betti table of a finite length module over $S$ exists as the Betti vector of a degree 0 differential module with finite length homology. This leads us to suspect a Boij-S\"oderberg type conjecture for differential modules.

\begin{conjecture}\label{conj:dmbs}
Every differential module over $S$ with homology of finite length has a Betti vector that can be expressed as a positive rational combination of Betti vectors of folds of pure resolutions with finite length homology whose degree sequences form a chain.
\end{conjecture}

Conjecture \ref{conj:dmbs} posits that the rational cone of Betti vectors of differential modules with finite length homology is the ``flattening" of the cone of Betti tables of finite length modules over $S$. That is, \emph{up to scalar multiple}, every Betti vector of a differential module with finite length homology can be obtained by taking the flattening the Betti table of a finite length module over $S$, as we saw in Example \ref{ex:cancel}. 
We return to this example now in light of Conjecture \ref{conj:dmbs} and compare the pure decomposition of the Betti vector to that of the Betti table of the homology.

\begin{example}\label{ex:pure-decomp}
    Let $R = \bbk[x,y]$ and $(D,\del)$ be as in Example \ref{ex:cancel}. The Betti table of the homology has pure decomposition
    \[
\begin{matrix}
       &0&1&2\\\text{0:}&1&2&1\\\text{1:}&\text{-}&\text
       {-}&\text{-}\\\text{2:}&1&2&1\\\end{matrix} \quad =\quad \begin{matrix}
       &0&1&2\\\text{0:}&1&2&1\\\text{1:}&\text{-}&\text
       {-}&\text{-}\\\text{2:}&\text{-}&\text
       {-}&\text{-}\\\end{matrix} \quad+\quad \begin{matrix}
       &0&1&2\\\text{0:}&\text{-}&\text
       {-}&\text{-}\\\text{1:}&\text{-}&\text
       {-}&\text{-}\\\text{2:}&1&2&1\\\end{matrix}
    \]
while the pure decomposition of $\betadm(D)$ is
    \[
\betadm(D) = \frac{1}{2}(2^\circ, 3, 0, 1, 0) + \frac{1}{2}(0^\circ, 1, 0, 3, 2)
    \]

    On the other hand, consider the differential module $(D', \del')$ where $D' = D$ is the same underlying module, but where $\del'$ is now given by the matrix 
\[
\begin{bmatrix}
       0 & 0 & x & y & 0 & 0 & 0 & 0\\
       0 & 0 & 0 & 0 & x & y & 0 & 0\\
       0 & 0 & 0 & 0 & 0 & 0 & -y & 0\\
       0 & 0 & 0 & 0 & 0 & 0 & x & 0\\
       0 & 0 & 0 & 0 & 0 & 0 & 0 & -y\\
       0 & 0 & 0 & 0 & 0 & 0 & 0 & x\\
       0 & 0 & 0 & 0 & 0 & 0 & 0 & 0\\
       0 & 0 & 0 & 0 & 0 & 0 & 0 & 0
    \end{bmatrix}
\]
(this matrix is \emph{almost} the same, except that the 1 has been replaced with a 0). The homology of $(D', \del')$ is the same as the homology of $(D, \del)$, but the Betti vector is 
\[
\betadm(D') = (1^\circ, 2, 2, 2, 1) = (1^\circ, 2, 1, 0, 0) + (0^\circ, 0, 1, 2, 1).
\]
\end{example}

In particular, the pure decomposition of the Betti vector of a differential module does not necessarily come from the pure decomposition of the Betti table of its homology. 

\subsection{Differential modules with positive degree}
We now return to the case where our differential modules have positive degree. As we will see, this case demonstrates behavior that differs starkly from that of $\DM(S,a)$ for $a\leq 0$, so much so that we may handle it with methods that are much more direct and elementary. We begin with an example of Remark \ref{rmk:apos}.

\begin{example}\label{ex:deg1koszul}
Let $R = \bbk[x,y]$ and consider the degree 1 differential module $R^4$ with differential 
\[
 \del = \begin{bmatrix}
0 & x & y & 0\\
0 & 0 & 0 & -y\\
0 & 0 & 0 & x\\
0 & 0 & 0 & 0
\end{bmatrix} \colon R^4\to R^4(1).
\]
This differential module is a minimal free flag, so its Betti numbers are the number of generators in each degree. That is, its Betti vector is $(4,0,0,\ldots)$.
\end{example}

One consequence of Example \ref{ex:deg1koszul} is that the rational cone of Betti vectors of degree 2 differential modules over $\bbk[x,y]$ with finite length homology includes the entire positive orthant. It turns out that this is not an isolated phenomenon.

\begin{proposition}\label{prop:wholeorthant}
For any $a>0$, the rational cone of Betti vectors of degree $a$ differential modules over $S$ with finite length homology is equal to the positive orthant in $\bbB$.
\end{proposition}

\begin{proof}
It suffices to show that the standard basis vectors are in the cone, i.e. that every standard basis vector is a rational multiple of the Betti vector of some $D\in\DM(S,a)$ with finite length homology. To show this, we show that we can get multiples of basis vectors as the degree $a$ folds of certain minimal free complexes.  

Let $\bfK$ be the koszul complex on $(x_1^a, x_2^a, \ldots, x_n^a)$, and define $K$ to be $\fold_a(\bfK)\in \DM(S,a)$. Then $K$ has underlying module $S^{2^n}$ and has minimal differential, so its Betti vector is $(2^n, 0, 0, \ldots)$. By twisting appropriately, we obtain a differential module whose Betti vector has $2^n$ in the $i^{th}$ entry and $0$ elsewhere.
\end{proof}

\begin{remark}
Any resolution may be folded into a degree $a$ differential module for any $a$, and the resulting Betti vector varies for different values of $a$. We can picture the vectors coming from pure resolutions as moving around while $a$ varies. For positive $a$, some of these vectors ``flare out" as far as possible resulting in a cone that fills the entire positive orthant, losing some of its combinatorial structure.
\end{remark}

As an immediate corollary of Proposition \ref{prop:wholeorthant}, we obtain the following:

\begin{theorem}
    Conjecture \ref{conj:dmbs} is true for $a>0$.
\end{theorem}
\begin{proof}
    We will prove this by constructing the extremal rays of the cone explicitly. By Proposition \ref{prop:wholeorthant}, extremal rays of the Betti cone are spanned by the standard basis vectors $\bfe_i$. The ray generated by $e_i$ contains the Betti vector of $\fold_a(\bfK(x_1^a,\ldots,x_n^a)[-i])$. Since the Koszul complex on $(x_1^a,\ldots, x_n^a)$ is pure of type $(0,a,2a,\ldots, na)$ with finite length homology, the Betti cone is spanned by Betti vectors of differential modules whose homology have pure resolutions. Furthermore, the shifted degree sequences of these resolutions form a chain.
\end{proof}

\section{Categorified Duality}\label{sec:duality}

In this section, we restrict to the case where $a = 0$. We give a generalization of the categorical pairing described in \cite{EE2017}. Our version is a pairing between the derived category of differential $S$-modules and the derived category of coherent sheaves on $\bbP^{n-1}$ and takes values in the derived category of differential $A$-modules where $A = \bbk[t]$.

\begin{definition}
By formally inverting all quasi-ismorphisms in $\DM(S,a)$, we form the \emph{derived category of differential $S$-modules}, denoted $\DDM(S,a)$. The subcategory of $\DDM(S,a)$ whose objects have finitely generated homology is the \emph{bounded derived category of differential $S$-modules}, denoted $\bDDM(S,a)$.
\end{definition}

\begin{notation}
We denote by $\DM(\bbP^{n-1})$ the category of sheaves of differential $\calO_{\bbP^{n-1}}$-modules, and denote the corresponding derived and bounded derived categories in the same way as above. We denote the bounded derived category of sheaves of (not necessarily differential) $\calO_{\bbP^{n-1}}$-modules by $\Db(\bbP^{n-1})$. \end{notation}

A key ingredient of the categorical pairing is the derived pushforard of a differential module, which we now describe here in some detail, following the same idea as \cite[pp. 11-12]{ToricTate}.

\begin{definition}[Derived pushforward of a differential module]\label{def:pshfwd}

Let $(\calD, \delta)$ be a representative of a class in $\bDDM(\bbP^{n-1}_A)$. Take $\pi\colon \bbP^{n-1}_A = \bbP^{n-1}\otimes\bbA^1 \to \bbA^1$ to be projection. 

    We exploit the equivalence in \cite[Remark 2.3]{ToricTate} between the category of differential modules over a ring $R$ and the category of 1-periodic $R$-complexes. Taking the expansion of $(\calD, \delta)$ we get a 1-periodic complex $\textbf{Ex}(\calD) = $ 

    \[
\cdots\leftarrow \calD\xleftarrow{\delta}\calD\xleftarrow{\delta}\calD\leftarrow\cdots
    \]

    Now we compute the derived pushforward $R\pi_*$ of this complex. In each position we take the \v{C}ech resolution of $\calD$ with respect to the usual affine cover of $\bbP^{n-1}$ which gives a bicomplex

\begin{figure}[H]
\begin{tikzcd}
\cdots & \calC^{n-1} \arrow[l, "\pm\delta"']       & \calC^{n-1} \arrow[l, "\pm\delta"']       & \calC^{n-1} \arrow[l, "\pm\delta"']       & \cdots \arrow[l, "\pm\delta"'] \\
\cdots & \vdots \arrow[l, "\delta"'] \arrow[u]  & \vdots \arrow[l, "\delta"'] \arrow[u]  & \vdots \arrow[l, "\delta"'] \arrow[u]  & \cdots \arrow[l, "\delta"'] \\
\cdots & \calC^1 \arrow[l, "-\delta"'] \arrow[u] & \calC^1 \arrow[l, "-\delta"'] \arrow[u] & \calC^1 \arrow[l, "-\delta"'] \arrow[u] & \cdots \arrow[l, "-\delta"'] \\
\cdots & \calC^0 \arrow[l, "\delta"'] \arrow[u] & \calC^0 \arrow[l, "\delta"'] \arrow[u] & \calC^0 \arrow[l, "\delta"'] \arrow[u] & \cdots \arrow[l, "\delta"']
\end{tikzcd}\end{figure}

where the vertical arrows are the maps in the \v{C}ech resolution and the horizontal arrows are induced by the differential $\delta$ (and thus still square to 0). We apply $\pi_*$ to this complex and totalize. Since the \v{C}ech complex is 0 after $n$ steps, each diagonal of the bicomplex above contains one copy each of $\calC^0, \calC^1, \ldots, \calC^{n-1}$, so this yields a 1-periodic complex

\[
\cdots\leftarrow\bigoplus_{i=0}^{n-1}\pi_*(\calC^i)\leftarrow\bigoplus_{i=0}^{n-1}\pi_*(\calC^{i})\leftarrow\bigoplus_{i=0}^{n-1}\pi_*( \calC^{i})\leftarrow \cdots
\]

We consider this as a sheaf of differential $\calO_{\bbA^{1}}$-modules, and the class represented by this object is what we define to be $R\pi_*(\calD)$.
\end{definition}

\begin{remark}
    The degree 0 condition is necessary in this definition of $R\pi_*$. In particular, we need the total complex after taking the \v{C}ech resolution to be one-periodic, but this fails to be the case if the twists of $\calD$ are not the same at each spot, since twisting a sheaf can genuinely change its cohomology.
\end{remark}

With this definition in hand, we now define our categorical pairing. Let $\varphi\colon S\to S\otimes_{\bbk} A$ be the homomorphism defined by $x_i\mapsto tx_i$. This defineds an $S$-module structure on $S\otimes_\bbk A = S[t]$ which in turn yields an action of $\bbP^{n-1}$ on $\bbP^{n-1}_A$. Using this action, we define for $F\in \DM(S)$, a bigraded sheaf of differential $\calO_{\bbP^{n-1}_A}$-modules $\Sigma F = \Tilde{F}\otimes^{DM}_{\bbP^{n-1}}\calO_{\bbP^{n-1}_A}$, where the tensor product is taken using the structure given by $\varphi$. If we write the differential $\del$ on $F$ as a matrix, the differental $\Sigma\del$ on $\Sigma F$ is given by the matrix obtained from $\del$ by replacing every entry $f$ by $t^{\deg(f)}f$. The purpose of this step is that it will allow us to push forward to $A$ while still retaining all of the graded data of our differential module over $S$.

\begin{definition}For $F\in\bDDM(S,0)$ and $\calE\in \Db(\bbP^{n-1})$, we define the functor \[\Phi\colon \bDDM(S,0)\times\Db(\bbP^{n-1})\to \bDDM(A,0)\] by
\[\Phi(F,\calE) = R\pi_*\left(\Sigma F\otimes^{DM}_{\bbP^{n-1}_A}(\calE\boxtimes\calO_{\bbA^1})\right).
\] where $\pi$ is the projection $\bbP^{n-1}\times\bbA^1\to\bbA^1$.
\end{definition}

\begin{definition}\label{def:absHilb}
    Let $\gamma(\calE)$ denote the vector whose $j\th$ entry is the sum of the ranks of the hypercohomology modules of $\calE(j)$. 
    We define the \emph{absolute Hilbert function} of $\calE$ to be the function $j\mapsto \gamma_j(\calE)$. This should be thought of as a cousin of the Hilbert polynomial of $\calE$, but where the Euler characteristic is replaced by the sum $\sum_{i\geq 0}h^i(\calE(j))$. 
\end{definition}

\begin{theorem}\label{thm:phiformula}
$\betadm(\Phi(F,\calE))$ depends only on $\betadm(F)$ and $\gamma(\calE)$, in particular the Betti vector of $\Phi(F,\calE)$ is given by the formula
\[
\betadm_j(\Phi(F,\calE)) = \betadm_j(F)\gamma_{-j}(\calE)
\]
\end{theorem}

\begin{proof}
Replacing $F$ with its minimal free resolution, we may assume that up to quasi-isomorphism that its underlying module is $F = \bigoplus_{j\in\bbZ}S(-j)^{\betadm(F)_j}$. Let $\calF = \Sigma F\otimes^{DM}_{\bbP^{n-1}_A} (\calE\boxtimes\calO_{\bbA^1})$. We can rewrite the underlying module of $\calF$ as
\begin{align*}
   \calF &= \Sigma F\otimes_{\bbP^{n-1}_A} (\calE\boxtimes\calO_{\bbA^1})\\
    &= \left(\bigoplus_{j\in\bbZ}\calO_{\bbP^{n-1}_A}(-j,-j)^{\betadm(F)_j}\right)\otimes_{\bbP^{n-1}_A}( \calE\boxtimes A)\\
    &=\bigoplus_{j\in\bbZ}\left(\calO_{\bbP^{n-1}}(-j)\boxtimes A(-j) \right)^{\betadm(F)_j}\otimes_{\bbP^{n-1}_A} (\calE\boxtimes A)\\
    &= \bigoplus_{j\in\bbZ}\left( \calE(-j)^{\betadm(F)_j}\boxtimes A(-j)\right)
\end{align*}

By Definition \ref{def:pshfwd}, to compute $R\pi_*(\calF)$ we expand $\calF$ to a one-periodic complex, resolve it with a \v{C}ech resolution at each step, then compute $\pi_*$. Letting $\calC^i(-)$ denote the $i\th$ sheaf in the \v{C}ech resolution using the standard affine cover of $\bbP^{n-1}$, we get the following bicomplex

\begin{figure}[H]
\begin{tikzcd}
\cdots & \bigoplus_{j\in\bbZ}\calC^{n-1}(\calE(-j)^{\betadm(F)_j})\boxtimes A(-j) \arrow[l, "\pm\delta"']       & \bigoplus_{j\in\bbZ}\calC^{n-1}(\calE(-j)^{\betadm(F)_j})\boxtimes A(-j)  \arrow[l, "\pm\delta"']       & \cdots \arrow[l, "\pm\delta"'] \\
 & \vdots  \arrow[u]  & \vdots  \arrow[u]    &  \\
\cdots & \bigoplus_{j\in\bbZ}\calC^{1}(\calE(-j)^{\betadm(F)_j})\boxtimes A(-j) \arrow[l, "-\delta"'] \arrow[u] & \bigoplus_{j\in\bbZ}\calC^{1}(\calE(-j)^{\betadm(F)_j})\boxtimes A(-j) \arrow[l, "-\delta"'] \arrow[u]  & \cdots \arrow[l, "-\delta"'] \\
\cdots & \bigoplus_{j\in\bbZ}\calC^{0}(\calE(-j)^{\betadm(F)_j})\boxtimes A(-j) \arrow[l, "\delta"'] \arrow[u] & \bigoplus_{j\in\bbZ}\calC^{0}(\calE(-j)^{\betadm(F)_j})\boxtimes A(-j)\arrow[l, "\delta"'] \arrow[u] &  \cdots \arrow[l, "\delta"']
\end{tikzcd}
\end{figure}

Now we apply $\pi_*$, but since every sheaf in the bicomplex is a box tensor of a sheaf on $\bbP^{n-1}$ with a twist of the structure sheaf on $\bbA^1$ and $\pi$ is a proper map, applying $\pi_*$ to $\calC^i(\calE(-j)^{\betadm(F)_j})\boxtimes A(-j)$ amounts to taking global sections of $\calC^i(\calE(-j)^{\betadm(F)_j})$ tensored with the appropriate twists of $A$. This turns each column into a direct sum of complexes $\Gamma\left(\bbP^{n-1}, \calC^\bullet(\calE(-j))\right)\otimes_\bbk A(-j)$. Since  $\Gamma\left(\bbP^{n-1}, \calC^\bullet(\calE(-j))\right)$ is a complex of $\bbk$-vector spaces, the vertical maps in the bicomplex are split. By \cite[Lemma 3.5]{Eisenbud_2003}, the totalization of the bicomplex is homotopic to 

\[
\cdots \leftarrow \bigoplus_{j\in\bbZ}H^{tot}(\calE(-j)^{\betadm(F)_j})\otimes_\bbk A(-j) \leftarrow \bigoplus_{j\in\bbZ}H^{tot}(\calE(-j)^{\betadm(F)_j})\otimes_\bbk A(-j) \leftarrow \cdots
\]
where $H^{tot}$ denotes the sum of the hypercohomology modules. The differential is a sum of maps $H^i\to H^{(j<i)}$, where each map comes from taking repeated compositions of the horizontal map $\delta$ with the section that splits the vertical map. In particular, this complex is minimal and 1-periodic. The minimal free differential $A$-module that this corresponds to is a representative of the class $R\pi_*(\calF)\in\bDDM(A,0)$, and we can compute its Betti vector by counting the number of generators of the underlying module in each degree. We have

\[
R\pi_*(\calF) \simeq \bigoplus_{j\in\bbZ}H^{tot}(\calE(-j)^{\betadm(F)_j})\otimes_\bbk A(-j) =
\bigoplus_{j\in\bbZ}A(-j)^{\betadm(F)_j\gamma(\calE)_{-j}}
\]

So the Betti vector of $\Phi(F,\calE)$ has $j\th$ entry 
\[
\betadm_j(\Phi(F,\calE)) = \betadm_j(F)\gamma_{-j}(\calE).
\]

\end{proof}

\begin{corollary}
Let $F\in \bDDM(S)$ and $\calE\in\Db(\bbP^{n-1})$. Then 
\begin{enumerate}
    \item $\betadm_j(F) = \betadm_j(\Phi(F, \calO_{\bbP^{n-1}}(j)).$
    \item $\gamma_j(\calE) = \betadm_{-j}(\Phi(S(j), \calE)).$
\end{enumerate}
\end{corollary}

\begin{lemma}\label{lem:sheafiso} 
As sheaves of differential $\calO_{\bbP^{n-1}_{\bbk(t)}}$-modules, there is an isomorphism 
\[
\left(\Sigma F\otimes_{\bbP^{n-1}_A}(\calE\boxtimes\calO_{\bbA^1})\right)\otimes_A \bbk(t) \simeq \left(\Tilde{F}\otimes_{\bbP^{n-1}_\bbk}\calE\right)\otimes_{\bbk}\bbk(t).
\]
\end{lemma}

\begin{proof}
We have a commutative diagram of rings

\begin{minipage}{7cm}
\begin{center}
\begin{tikzcd}
S \arrow[r, "\varphi"] \arrow[d, "1\otimes_{\bbk}\bbk(t)"'] & S\otimes_{\bbk}A \arrow[d, "1\otimes_A \bbk(t)"] \\
S\otimes_{\bbk}\bbk(t)  \arrow[r, "\tilde{\varphi}"']      & S\otimes_{\bbk}\bbk(t)                         
\end{tikzcd}
\end{center}
\end{minipage}
\begin{minipage}{7cm}
\begin{center}
\begin{align*}
    \varphi\colon x_i\mapsto x_i\otimes t\\
    \tilde{\varphi}\colon x_i\mapsto x_i\otimes t
\end{align*}
\end{center}
\end{minipage}

where the map $\tilde{\varphi}$ is invertible. This induces a diagram

\begin{center}
\begin{tikzcd}
\DM(\bbP^{n-1}_\bbk) \arrow[r, "\varphi"] \arrow[d, "-\otimes_{\bbk}\bbk(t)"'] & \DM(\bbP^{n-1}_A) \arrow[d, "-\otimes_A \bbk(t)"] \\
\DM(\bbP^{n-1}_{\bbk(t)})  \arrow[r, "\tilde{\varphi}"']      & \DM(\bbP^{n-1}_{\bbk(t)})                         
\end{tikzcd}
\end{center}

Unraveling definitions, we have  $\varphi\left(\tilde{F}\otimes_{\bbP^{n-1}}\calE \right) = \Sigma F\otimes_{\bbP^{n-1}_A}\left(\calE\boxtimes\calO_{\bbA^1}\right)$ as differential modules---that is, $\varphi$ is a map of modules that is compatible with the differential (if we consider the differentials to be given by matrices, the map $\varphi$ just replaces each entry $f$ with $t^{\deg f}f$). By commutativity of the diagram, the claim follows since $\tilde{\varphi}$ is invertible.
\end{proof}

\begin{theorem}\label{thm:exactfl}
If $\Tilde{F}\otimes \calE$ is exact then $\Phi(F,\calE)$ has finite length homology.
\end{theorem}

\begin{proof}
First we note that the conditions of exactness and finite length are not altered by the grading on $A$, so in what follows we will always forget the $A$-grading.
Because flat pullback commutes with proper pushforward \cite{stacks-project}, we have a commutative diagram

\begin{center}
\begin{tikzcd}
\bDDM(\bbP^{n-1}_A) \arrow[d, "-\otimes_A \bbk(t)"'] \arrow[r, "R\pi_*"] & \bDDM(A) \arrow[d, "-\otimes_\bbk \bbk(t)"] \\
\bDDM(\bbP^{n-1}_{\bbk(t)}) \arrow[r, "R\pi_*"]                         & \bDDM(\bbk(t))                            
\end{tikzcd}
\end{center}

The differential module $\Phi(F,\calE)$ is the image under $R\pi_*$ of the sheaf $\calF = \Sigma F\otimes_{\bbP^{n-1}_A}(\calE\boxtimes\calO_{\bbA^1})$. On the other hand, taking the tensor product $\calF\otimes_A\bbk(t)$ gives $\left(\Sigma F\otimes_{\bbP^{n-1}_A}(\calE\boxtimes\calO_{\bbA^1})\right)\otimes_A \bbk(t)$. By Lemma \ref{lem:sheafiso}, this is isomorphic to $(\tilde{F}\otimes\calE)\otimes_\bbk \bbk(t)$. But since $\tilde{F}\otimes\calE$ is exact, $(\tilde{F}\otimes\calE)\otimes_\bbk \bbk(t)$ is 0 in the derived category. Thus $\Phi(F,\calE)\otimes_\bbk \bbk(t)$ must be 0 in $\bDDM(\bbk(t))$. Since inverting $t$ yields a differential module with no homology, it therefore follows that $\Phi(F, \calE)$ has finite length homology.
\end{proof}

The implication of Theorems \ref{thm:phiformula} and \ref{thm:exactfl} is that the categorical pairing $\Phi$ descends to a pairing on cones

\[
\bsdm{S}{0}\times \cvb{\bbP^{n-1}}\to \bsdm{A}{0}
\]
where $\cvb{\bbP^{n-1}}$ is the cone of cohomology vectors of vector bundles on $\bbP^{n-1}$. In Section \ref{sec:1var}, we will describe the exterior facets of $\bsdm{A}{0}$ as linear functionals on $\bbB$, and in Section \ref{sec:predict} we will use the base case of $\bsdm{A}{0}$ alongside some computational data to outline a strategy for how this pairing on cones could be used to prove Conjecture \ref{conj:dmbs} in full for degree 0 differential modules.

We finish this section with a concrete example to demonstrate how $\Phi$ works.

\begin{example}
    Let $S = \bbk[x,y]$ and let $(F,\del)$ be the fold of the Koszul complex on the variables, so 
    \begin{align*}
        F = S\oplus S(-1)^2 \oplus S(-2), \quad \del = \begin{bmatrix}
            0 & x & y & 0\\
            0 & 0 & 0 & -y\\
            0 & 0 & 0 & x\\
            0 & 0 & 0 & 0
        \end{bmatrix}
    \end{align*}
    and let $\calE = \calO_{\bbP^1}$. Then $\Sigma F$ has underlying sheaf $\calO_{\bbP^1_A}\oplus \calO_{\bbP^1_A}(-1,-1)^2\oplus\calO_{\bbP^1_A}(-2,-2)$ with differential
    \[
\Sigma\del = \begin{bmatrix}
    0 & tx & ty & 0\\
    0 & 0 & 0 & -ty\\
    0 & 0 & 0 & tx\\
    0 & 0 & 0 & 0
\end{bmatrix}.
    \]
    Tensoring with $\calE\boxtimes\calO_{\bbA^1}$ does not change the underlying sheaf or differential since $\calE = \calO_{\bbP^1}$.

    Now computing $R\pi_*$, we get a differential $A$-module whose underying module is 
    \[
    \bigoplus_{j\in\bbZ} A(-j)^{\betadm(F)_j\gamma(\calE)_{-j}} = A\oplus A(-2)
    \]
    since $\calO_{\bbP^1}(-1)$ has no cohomology. To compute the differential on $\Phi(F,\calE)$, we use \cite[Lemma 3.5]{Eisenbud_2003}, which tells us that our differential should look like 
    \[
\begin{bmatrix}
    d_{0,0} & d_{1,0}\\
    0 & d_{1,1}
\end{bmatrix}
    \]
    where $d_{0,0}: A\to A$ and $d_{1,1}: A(-2)\to A(-2)$ are induced by the pieces of $\Sigma\del$ that map $\calO_{\bbP^1_A}\to\calO_{\bbP^1_A}$ and $\calO_{\bbP^1_A}(-2)\to\calO_{\bbP^1_A}(-2)$ and are therefore 0 since $\Sigma F$ was flag. By \ref{thm:exactfl} the homology is finite length, so the map $d_{1,0}: A(-2)\to A$ must be of the form $at^2$ for some unit $a\in A$, so our differential is
    \[
    \begin{bmatrix}
        0 & at^2\\
        0 & 0
    \end{bmatrix}
    \]
\end{example}

\section{Differential Modules Over $\bbk[t]$}\label{sec:1var}

In this section we investigate in detail the case of degree 0 differential modules over the polynomial ring in one variable. In particular, we prove Conjecture \ref{conj:dmbs} for differential modules over $\bbk[t]$ with finite length homology. When combined with the results of Section \ref{sec:duality}, this gives a base case for the theory in $n>1$ variables. 

\begin{theorem}\label{thm:1var}
Every degree 0 differential modules over $A = \bbk[t]$ with finite length homology has a Betti vector that can be expressed as a positive rational combination of Betti vectors of folds of pure resolutions whose degree sequences form a chain.
\end{theorem}

To prove Theorem \ref{thm:1var}, we will define three cones in $\bbB$. The first is the cone $B$ of Betti vectors of differential modules with finite length homology. The second is the cone $C$ spanned by Betti vectors of folds of pure resolutions. We will show these two cones are equal via a third cone $T$ defined by the nonnegativity of the following linear functionals:
\begin{equation}\label{eq:facets}
   \left\{ \tau_j \colon \beta \mapsto  -\beta_j + \sum_{i\neq j}\beta_i\quad\text{and}\quad
    \sigma_j \colon \beta\mapsto \beta_j \quad\text{for}\quad j\in\bbZ \right\} 
\end{equation}

Let $\bbB_{(p,q)}$ be the subspace of vectors $\bfb$ such that $b_n = 0$ for $n>p$ or $n<p$, i.e. all nonzero entries occur between the indices $p$ and $q$. We set $B_{(p,q)} = \bbB_{(p,q)}\cap B, C_{(p,q)} = \bbB_{(p,q)}\cap C,$ and $ T_{(p,q)} = \bbB_{(p,q)}\cap T$. We first show that for any $p,q$, $T_{(p,q)}\subseteq C_{(p,q)}$ by giving an explicit algorithm for decomposing vectors in $T$ as positive rational sums of Betti vectors of folds of pure resolutions.

\begin{algorithm}[H]\caption{Simplicial Decomposition Algorithm}\label{alg:decomp}
\begin{algorithmic}
\Require{vector $\bfb = (b_p, b_{p+1}, \ldots, b_{q})\in T$}
\Ensure{decomposition of $\bfb$ as rational combination of pure Betti vectors $c \bfe_{a,b}$}

\State{$L :=$  empty list}
\State{PHASE I:}
\While{$\tau_i( \bfb) > 0$ for all $p\leq i\leq q$}
    \State{Let $k,\ell$ be the first two indices such that $b_k, b_{\ell}\neq 0$}
    \State{Set $j$ to be the an index where $\tau_j(\bfb)$ is minimal}
    \State{$c := \min\{b_k, b_{\ell}, \frac{1}{2}\tau_j\cdot b\}$}
    \State{$\bfb :=\bfb-c\bfe_{\{k,\ell\}}$}
    \State{add $(c, \{ k,l \})$ to $L$}
\EndWhile
\State{PHASE II:}
\If{$\tau_j( \bfb) = 0$}
    \For{$i = p, \ldots , q$}
        \If{$b_i\neq 0$}
            \State{add $(b_i, \{i,j\})$ to $L$}
        \EndIf
        \State{$\bfb := \bfb - b_i\bfe_{\{i,j\}}$}
    \EndFor
\EndIf

\Return{$L$}
\end{algorithmic}
\end{algorithm}

We demonstrate the above algorithm with an example:

\begin{example}
    Take $\bfb = (3,4,2,5)$ and assume our nonzero window is between indicies $0$ and $3$. We can quickly see that $\tau_i(\bfb)>0$ for $i=0, \ldots, 3$, so we proceed with `Phase I', which is essentially a greedy algorithm. We take $(3,4,2,5)-(1,1,0,0) = (2,3,2,5)$, which is our new $\bfb$. Again we check that $\tau_i(\bfb)>0$ for each $i$. Since the first two entries are still nonzero, we take our new $\bfb$ to be $(2,3,2,5)-(1,1,0,0) = (1,2,2,5)$. Now $\tau_3(\bfb) = 0$, so we move on to `Phase II', which tells us to decompose $(1,2,2,5)$ as $(1,0,0,1) + 2(0,1,0,1) + 2(0,0,1,1)$. This gives us a final decomposition of $\bfb$ as
    \[
    \bfb = 2(1,1,0,0)+(1,0,0,1)+2(0,1,0,1)+2(0,0,1,1).
    \]
\end{example}

The correctness of the decomposition algorithm is proved in the following lemmas.
\begin{lemma}
If $\tau_j\bfb = 0$ for some $j$ and $b_n = 0$ for $k<n<l$ and $n<k$, then Phase II of the algorithm gives a decomposition of $\bfb$ as a rational combination of pure Betti vectors. 
\end{lemma}

\begin{proof}
By definition, $\tau_j\bfb = 0$ if and only if $b_j = \sum_{i\neq j}b_i$, so $\bfb = \sum_{i\neq j}b_i\bfe_{\{i,j\}}$.
\end{proof}

The next two lemmas combine to say that after each iteration of the while loop in Phase I of the algorithm, the vector $\bfb$ remains in the cone $T$.

\begin{lemma}
Let $\bfb\in T$ and $c = \min\{b_k, b_{\ell}, \frac{1}{2}\tau_j\cdot b\}$ where $j$ is an index such that $\tau_j\bfb$ is minimal. Let $\bfb' = \bfb - c\bfe_{\{k,l\}}$. Then $\tau_i\bfb' \geq 0$ for all $i$.
\end{lemma}

\begin{proof}
By definition, 
\begin{align*}
\tau_i\bfb' &= \tau_i\bfb - c\tau_i\bfe_{\{k,l\}}\\
&= \max\begin{cases}
        \tau_i\bfb - b_k\tau_i\bfe_{\{k,l\}}\\
        \tau_i\bfb - b_l\tau_i\bfe_{\{k,l\}}\\
        \tau_i\bfb - (\frac{1}{2}\tau_j\bfb)\tau_i\bfe_{\{k,l\}}
    \end{cases}
\end{align*}
Note that $\tau_i\bfe_{\{k,l\}} = 0$ if $i = k, l$ and $\tau_i\bfe_{\{k,l\}} = 2$ otherwise, so we have
\begin{align*}
\tau_i\bfb' = \begin{cases}
\tau_i\bfb &\text{ if } i = k,l \\
 \max\{\tau_i\bfb - 2b_k, \tau_i\bfb - 2b_l, \tau_i\bfb - \tau_j\bfb\} &\text{ if } i\neq k,l
\end{cases}
\end{align*}
By assumption, $\tau_i\bfb\geq 0$, so in the case where $i = k,l$ we're done. In the case where $i\neq k,l$, $\tau_i\bfb -\tau_j\bfb\geq 0$ since $j$ is an index where $\tau_j\bfb$ is minimal. Since $\tau_i\bfb -\tau_j\bfb\geq 0$, we must have that \[
\max\{\tau_i\bfb - 2b_k, \tau_i\bfb - 2b_l, \tau_i\bfb - \tau_j\bfb\}\geq 0.
\]
\end{proof}

\begin{lemma}
Let $\bfb\in T$ and $c = \min\{b_k, b_{\ell}, \frac{1}{2}\tau_j\cdot b\}$ where $j$ is an index such that $\tau_j\bfb$ is minimal. Let $\bfb' = \bfb - c\bfe_{\{k,l\}}$. Then $\sigma_i\bfb' \geq 0$ for all $i$.
\end{lemma}

\begin{proof}
\begin{align*}
\sigma_i\bfb' = \sigma_i\bfb - c\sigma_i\bfe_{\{k,l\}} = 
\begin{cases}
b_i-c &\text{ if } i=k,l\\
b_i &\text{ if } i\neq k,l
\end{cases}
\end{align*}
By assumption, $b_i\geq 0$, so in the case where $i\neq k,l$ $\sigma_i\bfb'\geq 0$. If $i = k,l$, 
\begin{align*}
    \sigma_i\bfb' = b_i - c = \max\{b_i-b_k, b_i-b_l, b_i - \frac{1}{2}\tau_j\bfb\}.
\end{align*}
If $i = k$, then $b_i - b_k = 0$ and if $i = l$ then $b_i - b_l = 0$, so in eather case $\sigma_i\bfb' = \max\{b_i-b_k, b_i-b_l, b_i - \frac{1}{2}\tau_j\bfb\}\geq 0$.
\end{proof}

We have seen already that Phase II of the algorithm terminates with a decomposition by pure vectors, so to show that the algorithm as a whole terminates we just need the following:

\begin{lemma}
Phase I of the algorithm terminates, that is, after a finite number of iterations, $\bfb$ will satisfy $\tau_j\bfb = 0$ for some $j$.
\end{lemma}

\begin{proof}
First notice that if we multiply $\bfb$ by a constant, the only impact on the algorithm is that at each step the choice of $c$ will be multiplied by said constant, no other aspect of the decomposition will be altered. Multiplying by a constant, we may assume that $b_i\in\bbZ$ for all $i$ and, furthermore, that $\sum b_i$ is even. This implies that $\frac{1}{2}\tau_i\bfb\in\bbZ$ for all $i$,so all entries of $\bfb$ are integers throughout every step of the algorithm.

We proceed by induction on the number of nonzero entries of $\bfb$. If all entries are 0, the result is trivial. Exactly one nonzero entry is impossible by the assumption that $\tau_i\bfb\geq 0$ for all $i$. If $b_k$ and $b_\ell$ are the only two nonzero entries of $\bfb$, then $\tau_k\bfb = \tau_\ell\bfb = 0$. Suppose that $\bfb$ has more than 2 nonzero entries and take $c = \min\{b_k, b_l, \frac{1}{2}\tau_j\bfb\}$ where $j$ is an index such that $\tau_j\bfb$ is minimal as in the algorithm. If $c = b_k$ or $b_l$, then replacing $\bfb$ by $\bfb - c\bfe_{\{k,l\}}$ results in either $b_l=0$ or $b_k=0$, so the number of nonzero entries decreases and we may apply the induction hypothesis to get the result. 

Now consider the case where $c = \frac{1}{2}\tau_j\bfb$ and let $\bfb' = \bfb - c\bfe_{\{k,l\}}$. Then we have 
\[
\tau_i\bfb' = \tau_i\bfb - \left(\frac{1}{2}\tau_j\bfb\right)\tau_i\bfe_{\{k,l\}} = 
\begin{cases}
\tau_i\bfb &\text{ if } i = k,l\\
\tau_i\bfb - \tau_j\bfb &\text{ if } i\neq k,l
\end{cases}.
\]
If $j\neq k,l$ then $\tau_j\bfb' = \tau_j\bfb - \tau_j\bfb = 0$ and we're done.

If $j = k$ or $l$, then $\tau_i\bfb' = \tau_i\bfb - \tau_j\bfb$, in particular $\tau_i\bfb$ strictly decreases for all $i\neq k,l$ every iteration. Since $\tau_i\bfb\in\bbZ$ for all $i$, after enough iterations there will be some $i\neq k,l$ for which $\tau_i\bfb'<\tau_j\bfb$. Since in each iteration, $j$ is picked so that $\tau_j$ is minimal, this means that in the subsequent iteration of the algorithm we will be in the case where $j\neq k,l$, so the result follows.   
\end{proof}

The above lemmas combine to show that for any $p<q$, we have containment $T_{(p,q)}\subseteq C_{(p,q)}$, i.e. that the cone spanned by the pure Betti vectors is contained in the cone defined by the facet equations in (\ref{eq:facets}). The following proposition says that additionally any vector in $T$ can be decomposed via pure Betti vectors whose degree sequences \emph{form a chain}.

\begin{proposition}\label{prop:chain}
The degree sequences in $L$ form a chain.
\end{proposition}

\begin{proof}
First we consider the two phases of the algorithm separately. 
In Phase I, the degree sequence added to the list at each iteration is $(k,\ell)$ where $k,\ell$ are the first two indices of nonzero entries in $\bfb$. In order for the values of $k,\ell$ to change at the next iteration, one of $b_k$ or $b_\ell$ must become zero, meaning that the first two indices of nonzero entries in $\bfb$ give a degree sequence strictly greater than the previous $(k,\ell)$. In Phase II, $j$ is fixed and $i$ increases, so the degree sequences $(i,j)$ for $i<j$ and $(j,i)$ for $i>j$ form a chain.

What remains to prove is that the first degree sequence added to $L$ in Phase II is greater than the last degree sequence added to $L$ in Phase I. Let $(k,\ell)$ be the final degree sequence added in Phase I. This means that for all $i<k$ and all $k<i<\ell$, $b_i = 0$ since $k$ and $\ell$ are the first two nonzero indices. If $\tau_j\cdot\bfb = 0$ then either $j>\ell$ or $\bfb = 0$. In the latter case, no more degree sequences are added, so we're done. In the former case, the first degree sequence added in Phase II is either $(i,j)$ for $i = k$ or $\ell$, in which case we have $(i,j)>(k,\ell)$ or $i>\ell$ in which case we also have $(i,j)>(k,\ell)$.
\end{proof}

We can now prove Theorem \ref{thm:1var} by establishing the equality of the three cones $B, C$, and $T$.

\begin{proof}[Proof of Theorem \ref{thm:1var}]
The containment $C\subseteq B$ follows from the existence of differential modules with pure Betti vectors.

By the simplicial decomposition algorithm and subsequent lemmas, we have $T_{(a,b)}\subseteq C_{(a,b)}$. Since any vector in $\bbB$ has finitely many nonzero entries, this gives the containment $T\subseteq C$. 

To show the containment $B\subseteq T$, we will prove that the functionals $\tau_j$ and $\sigma_j$ are nonnegative on the Betti vector $\betadm(D)$ for any graded, degree 0 differential module $D$ with finite length homology over $\bbk[t]$.
Let $D\in \DM(A,0)$ have finite length homology $H$. Since $H$ has finite length over a PID, it is torsion. Thus $H$ is of the form $\bigoplus A(-p)/(t^q)$, so it suffices to check nonnegativity for $H = A(-p)/(t^q)$. In this case, $H$ has minimal free resolution $A(-p-q)\xrightarrow{t^q}A(-p)\to 0$. By \cite[Theorem 3.2]{brown2021minimal}, $D$ has a free flag resolution $F$ of the form $A(-p-q)\oplus A(-p)$ with differential $\begin{bmatrix}0 & t^q\\ 0 & 0 \end{bmatrix}$, which is minimal. Thus $\betadm(D)_i = 1$ for $i = p, p+q$ and 0 otherwise, so $\tau_j\betadm(D) = 0$ or $2$ and $\sigma_j\betadm(D) = 0$ or $1$. 

This establishes the containments $B\subseteq T\subseteq C$ and $C\subseteq B$. Thus the cone $B$ of Betti vectors of differential modules with finite length homology over $A$ is equal to the cone $C$ spanned by the pure Betti vectors. 

Furthermore, by Proposition \ref{prop:chain}, every Betti vector of a differential module with finte length homology has a decomposition by pure Betti vectors whose degree sequences form a chain.
\end{proof}

\begin{remark}
For Betti diagrams of finite length modules, every Betti table has a \emph{unique} decomposition by pure Betti tables whose degree sequences form a chain. For differential modules however, the Betti vector my have a nonunique decomposition even after insisting that the pure vectors come from a chain of degree sequences. As an easy example, consider the Betti vector $(1,1,1,1)$. It can be decomposed either as $(1,1,0,0)+(0,0,1,1), (1,0,1,0)+(0,1,0,1), \text{ or } (1,0,0,1)+(0,1,1,0)$. Of these three decompositions, the first two both correspond to degree sequences that form a chain. Our algorithm would yield the first decomposition, corresponding to the chain of degree sequences $(0,1)<(2,3)$. The second decomposition corresponds to the chain $(0,2)<(1,3)$.
\end{remark}

\section{Predictions for Differential Modules Over $S$}\label{sec:predict}

We conjecture that Theorem \ref{thm:1var} gives a potential base case for the general theory in $n$ variables. Because the $\tau$ and $\sigma$ functionals are non-negative on differential modules in $\bDDM(A,0)$ with finite length homology, Theorem \ref{thm:exactfl} implies that they are non-negative on $\betadm(\Phi(F,\calE))$ for any coherent sheaf $\calE$ and any differential module $F\in \bDDM(S,0)$ with finite length homology. In the original theory, similar machinery provided a method for proving nonnegativity of facet equations since the facet equations of the cone spanned by the pure Betti tables came from cohomology tables of vector bundles on $\bbP^n$. Due to the apparent similarity between the original theory and the generalization to differential modules in the degree 0 case, we expect that something similar might hapen here as well. This illuminates a strategy for proving Conjecture \ref{conj:dmbs} in general. The idea is that once we know the exterior facets of the cone spanned by Betti vectors of differential modules with pure homology, this will enable us to prove non-negativity of the facet equations on the Betti vectors of differential modules in $\DM(S,0)$, provided the facet equations come from cohomology tables of sheaves on $\bbP^{n-1}$ as in the original theory. Proving a facet description has remained elusive thus far, but we can use Macaulay2 to compute facet descriptions of the cone in finite dimensional windows. 

For instance, let's examine the case where $S = \bbk[x,y,z]$ and look at the cone spanned by Betti vectors of differential modules with homology that is pure with degree sequence between $(0,1,2,3)$ and $(4,5,6,7)$. A Fourier Motzkin elimination computation gives the columns of the following matrix as a subset of the linear functionals in the facet description of the cone:

\begin{equation}\label{eq:facetmatrix}
\begin{bmatrix}
 & \vdots & \vdots & \vdots & \vdots & \vdots & \vdots & \vdots & \vdots & \vdots & \vdots & \vdots & \vdots &  \\
     \cdots & 15  & 15  & -10 & 5 & 6 & 15 & 2 & 3 & 5 & 1  & 8 & 3 & \cdots \\
     \cdots & 10 & 10 & 6  & 2 & -3 & -8 & 0 & 1 &0 & 0   & 3 & 0 &\cdots \\
     \cdots &  -6   & 6  & 3  & 0  & 1  & 3  & 1 & 0  & 3 & 0  & 0  & 1 & \cdots \\
     \cdots & 3 & -3   & 1  & 1 & 0 & 0   & 1 & 0  & -4 & 1  & 1 & 0  & \cdots \\
     \cdots & 1  & 1 & 0  & 1 & 0  & 1 & 0  & 1 & 3 & -3 & 0  & 3  & \cdots \\
     \cdots &  0  & 0  & 0 &  0 & 1 & 0  & 2 & -3 & 0  & 6  & 3 & 8  & \cdots \\
     \cdots & 0  & 0   & 1  & 2 & 3 & 3  & 5 & 6 & 5 & 10 & 8 & -15 & \cdots \\
     \cdots & 1  & 1 & 3  & -5 & 6 & 8  & -9 & 10 & 12 & 15 & -15 & 24 & \cdots \\
 & \vdots & \vdots & \vdots & \vdots & \vdots & \vdots & \vdots & \vdots & \vdots & \vdots & \vdots & \vdots & 
\end{bmatrix}
\end{equation}

One might recognize the entries of each column as the ranks of the sheaf cohomology modules of \emph{supernatural} vector bundles on $\bbP^2$ (see \cite{eisenbud2009cohomology}) appearing with a single negative entry in each column. For instance, twists of the line bundle $\calO_{\bbP^2}$ have cohomology modules whose ranks are given by the entries in the first three columns of (\ref{eq:facetmatrix}), with some shift. If we look at the sign pattern, we notice that there is a single negative entry in each column, which is exactly what we saw in the facet equations for $\bsdm{A}{0}$, where the facets were given by taking inner product with vectors of the form $\tau_j = (\ldots, 1, 1, -1, 1, \ldots )$, with the negative entry occurring in the $j\th$ entry. What is more, we see this same phenomenon repeated for all examples computed. 

This leads us to a conjecture that, similarly to \cite{EE2017} the facets $\bsdm{S}{0}$ come from combining the cohomology tables of certain vector bundles on $\bbP^{n-1}$ with the facets for $\bsdm{A}{0}$. 

\begin{conjecture}
The exterior facets of the cone of Betti vectors of degree 0 differential modules with finite length homology over $S$ are given by the vanishing of linear functionals that arise as the composition
\[
\bDDM(S)\xrightarrow{\Phi(-,\calE)} \bDDM(A)\xrightarrow{f}\bbZ
\]
where $\calE$ is a supernatural vector bundle on $\bbP^{n-1}$ and $f$ is one of the linear functionals $\tau_j$ or $\sigma_i$ defined in (\ref{eq:facets}).
\end{conjecture}

To see this in action, consider the first column of (\ref{eq:facetmatrix}), and suppose we index so that the $-6$ is the entry in position 0. The vector 
\[
(\ldots, 15, 10, 6, 3, 1, 0, 0, 1, \ldots)
\]
is equal to $\gamma(\calO_{\bbP^{n-1}}(-5))$. Dot product with this vector defines a linear functional that sends $\betadm(F)$ to $\betadm(\Phi(F,\calO_{\bbP^{n-1}}(-5)))$. Composing this with $\tau_0$ as defined in (\ref{eq:facets}) gives the linear function that is dot product with the first column of (\ref{eq:facetmatrix}).

\bibliographystyle{amsalpha}
\bibliography{bibliography}

\providecommand{\bysame}{\leavevmode\hbox to3em{\hrulefill}\thinspace}
\providecommand{\MR}{\relax\ifhmode\unskip\space\fi MR }
\providecommand{\MRhref}[2]{%
  \href{http://www.ams.org/mathscinet-getitem?mr=#1}{#2}
}
\providecommand{\href}[2]{#2}
\begin{thebibliography}{BBEG12}

\bibitem[ABI07]{avramov2007class}
Luchezar~L. Avramov, Ragnar-Olaf Buchweitz, and Srikanth Iyengar, \emph{Class
  and rank of differential modules}, Invent. Math. \textbf{169} (2007), no.~1,
  1--35. \MR{2308849}

\bibitem[BBEG12]{Berkesch_2012}
Christine Berkesch, Jesse Burke, Daniel Erman, and Courtney Gibbons, \emph{The
  cone of betti diagrams over a hypersurface ring of low embedding dimension},
  Journal of Pure and Applied Algebra \textbf{216} (2012), no.~10, 2256--2268.

\bibitem[BD10]{boocher2010rank}
Adam Boocher and Justin~W DeVries, \emph{On the rank of multigraded
  differential modules}, arXiv preprint arXiv:1011.2167 (2010).

\bibitem[BE21a]{brown2021minimal}
Michael~K. Brown and Daniel Erman, \emph{Minimal free resolutions of
  differential modules}, 2021.

\bibitem[BE21b]{ToricTate}
Michael~K. Brown and Daniel Erman, \emph{Tate resolutions on toric varieties},
  2021.

\bibitem[BEKS11]{Berkesch_2011}
Christine Berkesch, Daniel Erman, Manoj Kummini, and Steven~V. Sam,
  \emph{Shapes of free resolutions over a local ring}, Mathematische Annalen
  \textbf{354} (2011), no.~3, 939--954.

\bibitem[BS08]{2008}
Mats Boij and Jonas Söderberg, \emph{Graded betti numbers of cohen-macaulay
  modules and the multiplicity conjecture}, Journal of the London Mathematical
  Society \textbf{78} (2008), no.~1, 85–106.

\bibitem[BS12]{boij2012betti}
Mats Boij and Jonas S{\"o}derberg, \emph{Betti numbers of graded modules and
  the multiplicity conjecture in the non-cohen--macaulay case}, Algebra \&
  Number Theory \textbf{6} (2012), no.~3, 437--454.

\bibitem[BS15]{Boij_2015}
Mats Boij and Gregory~G. Smith, \emph{Cones of hilbert functions},
  International Mathematics Research Notices \textbf{2015} (2015), no.~20,
  10314--10338.

\bibitem[BV22]{BV22}
Maya Banks and Keller VandeBogert, \emph{Differential modules with complete
  intersection homology}, 2022.

\bibitem[CE16]{cartan2016homological}
Henry Cartan and Samuel Eilenberg, \emph{Homological algebra (pms-19), volume
  19}, Princeton university press, 2016.

\bibitem[EE17]{EE2017}
David Eisenbud and Daniel Erman, \emph{Categorified duality in
  boij–söderberg theory and invariants of free complexes}, Journal of the
  European Mathematical Society \textbf{19} (2017), no.~9, 2657–2695.

\bibitem[EFS03]{Eisenbud_2003}
David Eisenbud, Gunnar Fl{\o}ystad, and Frank-Olaf Schreyer, \emph{Sheaf
  cohomology and free resolutions over exterior algebras}, Transactions of the
  American Mathematical Society \textbf{355} (2003), no.~11, 4397--4426.

\bibitem[Erm09]{2009}
Daniel Erman, \emph{The semigroup of betti diagrams}, Algebra \& Number Theory
  \textbf{3} (2009), no.~3, 341–365.

\bibitem[ES]{eisenbud2008Betti}
David Eisenbud and Frank-Olaf Schreyer, \emph{Betti numbers of graded modules
  and cohomology of vector bundles}, J. Amer. Math. Soc. \textbf{22}, no.~3,
  859--888.

\bibitem[ES09]{eisenbud2009cohomology}
David Eisenbud and Frank-Olaf Schreyer, \emph{Cohomology of coherent sheaves
  and series of supernatural bundles}, 2009.

\bibitem[ES16]{Erman_2016}
Daniel Erman and Steven~V Sam, \emph{Supernatural analogues of beilinson
  monads}, Compositio Mathematica \textbf{152} (2016), no.~12, 2545--2562.

\bibitem[FJK11]{2011}
Gunnar Fl{\o}ystad, Trygve Johnsen, and Andreas~Leopold Knutsen (eds.),
  \emph{Combinatorial aspects of commutative algebra and algebraic geometry},
  Springer Berlin Heidelberg, 2011.

\bibitem[FL18]{Ford_2018b}
Nicolas Ford and Jake Levinson, \emph{Foundations of
  boij{\textendash}söderberg theory for grassmannians}, Compositio Mathematica
  \textbf{154} (2018), no.~10, 2205--2238.

\bibitem[Fl{\o}10]{Floystad_2010}
Gunnar Fl{\o}ystad, \emph{The linear space of betti diagrams of multigraded
  artinian modules}, Mathematical Research Letters \textbf{17} (2010), no.~5,
  943--958.

\bibitem[FLS18]{Ford_2018a}
Nicolas Ford, Jake Levinson, and Steven Sam, \emph{Towards
  boij{\textendash}söderberg theory for grassmannians : the case of square
  matrices}, Algebra \& Number Theory \textbf{12} (2018), no.~2, 285--303.

\bibitem[GS]{M2}
Daniel~R. Grayson and Michael~E. Stillman, \emph{Macaulay2, a software system
  for research in algebraic geometry}, Available at
  url{http://www.math.uiuc.edu/Macaulay2/}.

\bibitem[Hoc74]{Hoc74}
Melvin Hochster, \emph{{The equicharacteristic case of some homological
  conjectures on local rings}}, Bulletin of the American Mathematical Society
  \textbf{80} (1974), no.~4, 683 -- 686.

\bibitem[IMW22]{IMW22}
Srikanth~B. Iyengar, Linquan Ma, and Mark~E. Walker, \emph{Lim ulrich sequences
  and boij-söderberg cones}, 2022.

\bibitem[IW18]{iyengar2018examples}
Srikanth~B Iyengar and Mark~E Walker, \emph{Examples of finite free complexes
  of small rank and small homology}, Acta Mathematica \textbf{221} (2018),
  no.~1, 143--158.

\bibitem[PS73]{PS73}
Christian Peskine and Lucien Szpiro, \emph{Dimension projective finie et
  cohomologie locale}, Publications Math\'ematiques de l'IH\'ES \textbf{42}
  (1973), 47--119 (fr).

\bibitem[Rob89]{Rob89}
Paul Roberts, \emph{Intersection theorems}, Commutative Algebra (New York, NY)
  (M.~Hochster, C.~Huneke, and J.~D. Sally, eds.), Springer New York, 1989,
  pp.~417--436.

\bibitem[SS21]{DeStefani2021}
Alessandro~De Stefani and Ilya Smirnov, \emph{Decomposition of graded local
  cohomology tables}, Mathematische Zeitschrift \textbf{297} (2021), no.~1,
  1--24.

\bibitem[{Sta}22]{stacks-project}
The {Stacks project authors}, \emph{The stacks project},
  \url{https://stacks.math.columbia.edu}, 2022.

\bibitem[{\c{S}}{\"U}19]{csenturk2019carlsson}
Berrin {\c{S}}ent{\"u}rk and {\"O}zg{\"u}n {\"U}nl{\"u}, \emph{Carlsson's rank
  conjecture and a conjecture on square-zero upper triangular matrices},
  Journal of Pure and Applied Algebra \textbf{223} (2019), no.~6, 2562--2584.

\end{thebibliography}
\addcontentsline{toc}{section}{Bibliography}

\end{document}